\documentclass[11pt]{amsart}

\usepackage{amsmath}
\usepackage{amssymb}
\usepackage{amscd}
\usepackage{color}
\usepackage{hyperref}
\usepackage{mathrsfs}

\usepackage{xy}
\xyoption{all}

\newcommand{\nc}{\newcommand}

\nc{\rJ}{{\mathrm{J}}}

\nc{\CC}{{\mathbb{C}}}
\nc{\LL}{{\mathbb{L}}}
\nc{\RR}{{\mathbb{R}}}
\nc{\PP}{{\mathbb{P}}}
\nc{\OO}{{\mathbb{O}}}

\nc{\QQ}{{\mathbb{Q}}}
\nc{\ZZ}{{\mathbb{Z}}}

\nc{\cA}{{\mathscr{A}}}
\nc{\cB}{{\mathscr{B}}}
\nc{\cC}{{\mathscr{C}}}
\nc{\cD}{{\mathscr{D}}}
\nc{\cE}{{\mathscr{E}}}
\nc{\cF}{{\mathscr{F}}}
\nc{\cG}{{\mathscr{G}}}
\nc{\cH}{{\mathscr{H}}}
\nc{\cI}{{\mathscr{I}}}
\nc{\cJ}{{\mathscr{J}}}
\nc{\cK}{{\mathscr{K}}}
\nc{\cL}{{\mathscr{L}}}
\nc{\cM}{{\mathscr{M}}}
\nc{\cN}{{\mathscr{N}}}
\nc{\cO}{{\mathscr{O}}}
\nc{\cP}{{\mathscr{P}}}
\nc{\cQ}{{\mathscr{Q}}}
\nc{\cR}{{\mathscr{R}}}
\nc{\cS}{{\mathscr{S}}}
\nc{\cT}{{\mathscr{T}}}
\nc{\cU}{{\mathscr{U}}}
\nc{\cV}{{\mathscr{V}}}
\nc{\cW}{{\mathscr{W}}}
\nc{\cX}{{\mathscr{X}}}
\nc{\cY}{{\mathscr{Y}}}
\nc{\cZ}{{\mathscr{Z}}}

\nc{\bA}{{\mathbf{A}}}
\nc{\bB}{{\mathbf{B}}}
\nc{\bC}{{\mathbf{C}}}
\nc{\bD}{{\mathbf{D}}}
\nc{\bE}{{\mathbf{E}}}
\nc{\bF}{{\mathbf{F}}}
\nc{\bG}{{\mathbf{G}}}
\nc{\bH}{{\mathbf{H}}}
\nc{\bI}{{\mathbf{I}}}
\nc{\bJ}{{\mathbf{J}}}
\nc{\bK}{{\mathbf{K}}}
\nc{\bL}{{\mathbf{L}}}
\nc{\bM}{{\mathbf{M}}}
\nc{\bN}{{\mathbf{N}}}
\nc{\bO}{{\mathbf{O}}}
\nc{\bP}{{\mathbf{P}}}
\nc{\bQ}{{\mathbf{Q}}}
\nc{\bR}{{\mathbf{R}}}
\nc{\bS}{{\mathbf{S}}}
\nc{\bT}{{\mathbf{T}}}
\nc{\bU}{{\mathbf{U}}}
\nc{\bV}{{\mathbf{V}}}
\nc{\bW}{{\mathbf{W}}}
\nc{\bX}{{\mathbf{X}}}
\nc{\bY}{{\mathbf{Y}}}
\nc{\bZ}{{\mathbf{Z}}}

\nc{\ba}{{\mathbf{a}}}
\nc{\bb}{{\mathbf{b}}}
\nc{\bc}{{\mathbf{c}}}
\nc{\bd}{{\mathbf{d}}}
\nc{\be}{{\mathbf{e}}}
\nc{\bg}{{\mathbf{g}}}
\nc{\bh}{{\mathbf{h}}}
\nc{\bi}{{\mathbf{i}}}
\nc{\bj}{{\mathbf{j}}}
\nc{\bk}{{\mathbf{k}}}
\nc{\bl}{{\mathbf{l}}}
\nc{\bm}{{\mathbf{m}}}
\nc{\bn}{{\mathbf{n}}}
\nc{\bo}{{\mathbf{o}}}
\nc{\bp}{{\mathbf{p}}}
\nc{\bq}{{\mathbf{q}}}
\nc{\br}{{\mathbf{r}}}
\nc{\bs}{{\mathbf{s}}}
\nc{\bt}{{\mathbf{t}}}
\nc{\bu}{{\mathbf{u}}}
\nc{\bv}{{\mathbf{v}}}
\nc{\bw}{{\mathbf{w}}}
\nc{\bx}{{\mathbf{x}}}
\nc{\by}{{\mathbf{y}}}
\nc{\bz}{{\mathbf{z}}}

\nc{\fA}{{\mathfrak{A}}}
\nc{\fB}{{\mathfrak{B}}}
\nc{\fC}{{\mathfrak{C}}}
\nc{\fD}{{\mathfrak{D}}}
\nc{\fE}{{\mathfrak{E}}}
\nc{\fF}{{\mathfrak{F}}}
\nc{\fG}{{\mathfrak{G}}}
\nc{\fH}{{\mathfrak{H}}}
\nc{\fI}{{\mathfrak{I}}}
\nc{\fJ}{{\mathfrak{J}}}
\nc{\fK}{{\mathfrak{K}}}
\nc{\fL}{{\mathfrak{L}}}
\nc{\fM}{{\mathfrak{M}}}
\nc{\fN}{{\mathfrak{N}}}
\nc{\fO}{{\mathfrak{O}}}
\nc{\fP}{{\mathfrak{P}}}
\nc{\fQ}{{\mathfrak{Q}}}
\nc{\fR}{{\mathfrak{R}}}
\nc{\fS}{{\mathfrak{S}}}
\nc{\fT}{{\mathfrak{T}}}
\nc{\fU}{{\mathfrak{U}}}
\nc{\fV}{{\mathfrak{V}}}
\nc{\fW}{{\mathfrak{W}}}
\nc{\fX}{{\mathfrak{X}}}
\nc{\fY}{{\mathfrak{Y}}}
\nc{\fZ}{{\mathfrak{Z}}}

\nc{\fa}{{\mathfrak{a}}}
\nc{\fb}{{\mathfrak{b}}}
\nc{\fc}{{\mathfrak{c}}}
\nc{\fd}{{\mathfrak{d}}}
\nc{\fe}{{\mathfrak{e}}}
\nc{\ff}{{\mathfrak{f}}}
\nc{\fg}{{\mathfrak{g}}}
\nc{\fh}{{\mathfrak{h}}}
\nc{\fj}{{\mathfrak{j}}}
\nc{\fk}{{\mathfrak{k}}}
\nc{\fl}{{\mathfrak{l}}}
\nc{\fm}{{\mathfrak{m}}}
\nc{\fn}{{\mathfrak{n}}}
\nc{\fo}{{\mathfrak{o}}}
\nc{\fp}{{\mathfrak{p}}}
\nc{\fq}{{\mathfrak{q}}}
\nc{\fr}{{\mathfrak{r}}}
\nc{\fs}{{\mathfrak{s}}}
\nc{\ft}{{\mathfrak{t}}}
\nc{\fu}{{\mathfrak{u}}}
\nc{\fv}{{\mathfrak{v}}}
\nc{\fw}{{\mathfrak{w}}}
\nc{\fx}{{\mathfrak{x}}}
\nc{\fy}{{\mathfrak{y}}}
\nc{\fz}{{\mathfrak{z}}}

\nc{\sA}{{\mathsf{A}}}
\nc{\sB}{{\mathsf{B}}}
\nc{\sC}{{\mathsf{C}}}
\nc{\sD}{{\mathsf{D}}}
\nc{\sE}{{\mathsf{E}}}
\nc{\sF}{{\mathsf{F}}}
\nc{\sG}{{\mathsf{G}}}
\nc{\sH}{{\mathsf{H}}}
\nc{\sI}{{\mathsf{I}}}
\nc{\sJ}{{\mathsf{J}}}
\nc{\sK}{{\mathsf{K}}}
\nc{\sL}{{\mathsf{L}}}
\nc{\sM}{{\mathsf{M}}}
\nc{\sN}{{\mathsf{N}}}
\nc{\sO}{{\mathsf{O}}}
\nc{\sP}{{\mathsf{P}}}
\nc{\sQ}{{\mathsf{Q}}}
\nc{\sR}{{\mathsf{R}}}
\nc{\sS}{{\mathsf{S}}}
\nc{\sT}{{\mathsf{T}}}
\nc{\sU}{{\mathsf{U}}}
\nc{\sV}{{\mathsf{V}}}
\nc{\sW}{{\mathsf{W}}}
\nc{\sX}{{\mathsf{X}}}
\nc{\sY}{{\mathsf{Y}}}
\nc{\sZ}{{\mathsf{Z}}}

\nc{\sa}{{\mathsf{a}}}
\nc{\sd}{{\mathsf{d}}}
\nc{\se}{{\mathsf{e}}}
\nc{\sg}{{\mathsf{g}}}
\nc{\sh}{{\mathsf{h}}}
\nc{\si}{{\mathsf{i}}}
\nc{\sj}{{\mathsf{j}}}
\nc{\sk}{{\mathsf{k}}}
\nc{\sm}{{\mathsf{m}}}
\nc{\sn}{{\mathsf{n}}}
\nc{\so}{{\mathsf{o}}}
\nc{\sq}{{\mathsf{q}}}
\nc{\sr}{{\mathsf{r}}}
\nc{\st}{{\mathsf{t}}}
\nc{\su}{{\mathsf{u}}}
\nc{\sv}{{\mathsf{v}}}
\nc{\sw}{{\mathsf{w}}}
\nc{\sx}{{\mathsf{x}}}
\nc{\sy}{{\mathsf{y}}}
\nc{\sz}{{\mathsf{z}}}

\nc{\oA}{{\overline{A}}}
\nc{\oB}{{\overline{B}}}
\nc{\oC}{{\overline{C}}}
\nc{\oD}{{\overline{D}}}
\nc{\oE}{{\overline{E}}}
\nc{\oF}{{\overline{F}}}
\nc{\oG}{{\overline{G}}}
\nc{\oH}{{\overline{H}}}
\nc{\oI}{{\overline{I}}}
\nc{\oJ}{{\overline{J}}}
\nc{\oK}{{\overline{K}}}
\nc{\oL}{{\overline{L}}}
\nc{\oM}{{\overline{M}}}
\nc{\oN}{{\overline{N}}}
\nc{\oO}{{\overline{O}}}
\nc{\oP}{{\overline{P}}}
\nc{\oQ}{{\overline{Q}}}
\nc{\oR}{{\overline{R}}}
\nc{\oS}{{\overline{S}}}
\nc{\oT}{{\overline{T}}}
\nc{\oU}{{\overline{U}}}
\nc{\oV}{{\overline{V}}}
\nc{\oW}{{\overline{W}}}
\nc{\oX}{{\overline{X}}}
\nc{\oY}{{\overline{Y}}}
\nc{\oZ}{{\overline{Z}}}

\nc{\oa}{{\overline{a}}}
\nc{\ob}{{\overline{b}}}
\nc{\oc}{{\overline{c}}}
\nc{\od}{{\overline{d}}}
\nc{\of}{{\overline{f}}}
\nc{\og}{{\overline{g}}}
\nc{\oh}{{\overline{h}}}
\nc{\oi}{{\overline{i}}}
\nc{\oj}{{\overline{j}}}
\nc{\ok}{{\overline{k}}}
\nc{\ol}{{\overline{l}}}
\nc{\om}{{\overline{m}}}
\nc{\on}{{\overline{n}}}
\nc{\oo}{{\overline{o}}}
\nc{\op}{{\overline{p}}}
\nc{\oq}{{\overline{q}}}
\nc{\os}{{\overline{s}}}
\nc{\ot}{{\overline{t}}}
\nc{\ou}{{\overline{u}}}
\nc{\ov}{{\overline{v}}}
\nc{\ow}{{\overline{w}}}
\nc{\ox}{{\overline{x}}}
\nc{\oy}{{\overline{y}}}
\nc{\oz}{{\overline{z}}}

\nc{\tA}{{\tilde{A}}}
\nc{\tB}{{\tilde{B}}}
\nc{\tC}{{\tilde{C}}}
\nc{\tD}{{\tilde{D}}}
\nc{\tE}{{\tilde{E}}}
\nc{\tF}{{\tilde{F}}}
\nc{\tG}{{\tilde{G}}}
\nc{\tH}{{\tilde{H}}}
\nc{\tI}{{\tilde{I}}}
\nc{\tJ}{{\tilde{J}}}
\nc{\tK}{{\tilde{K}}}
\nc{\tL}{{\tilde{L}}}
\nc{\tM}{{\tilde{M}}}
\nc{\tN}{{\tilde{N}}}
\nc{\tO}{{\tilde{O}}}
\nc{\tP}{{\tilde{P}}}
\nc{\tQ}{{\tilde{Q}}}
\nc{\tR}{{\tilde{R}}}
\nc{\tS}{{\tilde{S}}}
\nc{\tT}{{\tilde{T}}}
\nc{\tU}{{\tilde{U}}}
\nc{\tV}{{\tilde{V}}}
\nc{\tW}{{\tilde{W}}}
\nc{\tX}{{\tilde{X}}}
\nc{\tY}{{\tilde{Y}}}
\nc{\tZ}{{\tilde{Z}}}

\nc{\ta}{{\tilde{a}}}
\nc{\tb}{{\tilde{b}}}
\nc{\tc}{{\tilde{c}}}
\nc{\td}{{\tilde{d}}}
\nc{\te}{{\tilde{e}}}
\nc{\tf}{{\tilde{f}}}
\nc{\tg}{{\tilde{g}}}
\nc{\ti}{{\tilde{i}}}
\nc{\tj}{{\tilde{j}}}
\nc{\tk}{{\tilde{k}}}
\nc{\tl}{{\tilde{l}}}
\nc{\tm}{{\tilde{m}}}
\nc{\tn}{{\tilde{n}}}
\nc{\tp}{{\tilde{p}}}
\nc{\tq}{{\tilde{q}}}
\nc{\tr}{{\tilde{r}}}
\nc{\ts}{{\tilde{s}}}
\nc{\tu}{{\tilde{u}}}
\nc{\tv}{{\tilde{v}}}
\nc{\tw}{{\tilde{w}}}
\nc{\tx}{{\tilde{x}}}
\nc{\ty}{{\tilde{y}}}
\nc{\tz}{{\tilde{z}}}

\nc{\hA}{{\hat{A}}}
\nc{\hB}{{\hat{B}}}
\nc{\hC}{{\hat{C}}}
\nc{\hD}{{\hat{D}}}
\nc{\hE}{{\hat{E}}}
\nc{\hF}{{\hat{F}}}
\nc{\hG}{{\hat{G}}}
\nc{\hH}{{\hat{H}}}
\nc{\hI}{{\hat{I}}}
\nc{\hJ}{{\hat{J}}}
\nc{\hK}{{\hat{K}}}
\nc{\hL}{{\hat{L}}}
\nc{\hM}{{\hat{M}}}
\nc{\hN}{{\hat{N}}}
\nc{\hO}{{\hat{O}}}
\nc{\hP}{{\hat{P}}}
\nc{\hQ}{{\hat{Q}}}
\nc{\hR}{{\hat{R}}}
\nc{\hS}{{\hat{S}}}
\nc{\hT}{{\hat{T}}}
\nc{\hU}{{\hat{U}}}
\nc{\hV}{{\hat{V}}}
\nc{\hW}{{\hat{W}}}
\nc{\hX}{{\hat{X}}}
\nc{\hY}{{\hat{Y}}}
\nc{\hZ}{{\hat{Z}}}

\nc{\ha}{{\hat{a}}}
\nc{\hb}{{\hat{b}}}
\nc{\hc}{{\hat{c}}}
\nc{\hd}{{\hat{d}}}
\nc{\he}{{\hat{e}}}
\nc{\hf}{{\hat{f}}}
\nc{\hg}{{\hat{g}}}
\nc{\hh}{{\hat{h}}}
\nc{\hi}{{\hat{i}}}
\nc{\hj}{{\hat{j}}}
\nc{\hk}{{\hat{k}}}
\nc{\hl}{{\hat{l}}}
\nc{\hm}{{\hat{m}}}
\nc{\hn}{{\hat{n}}}
\nc{\ho}{{\hat{o}}}
\nc{\hp}{{\hat{p}}}
\nc{\hq}{{\hat{q}}}
\nc{\hr}{{\hat{r}}}
\nc{\hs}{{\hat{s}}}
\nc{\hu}{{\hat{u}}}
\nc{\hv}{{\hat{v}}}
\nc{\hw}{{\hat{w}}}
\nc{\hx}{{\hat{x}}}
\nc{\hy}{{\hat{y}}}
\nc{\hz}{{\hat{z}}}

\nc{\eps}{\varepsilon}
\nc{\lan}{\big\langle}
\nc{\ran}{\big\rangle}
\nc{\kk}{{\mathsf{k}}}

\def\bw#1#2{\textstyle{\bigwedge\hskip-0.9mm^{#1}}\hskip0.2mm{#2}}

\DeclareMathOperator{\Bl}{\mathrm{Bl}}

\nc{\Sym}{\mathrm{S}}

\DeclareMathOperator{\Gr}{\mathrm{Gr}}

\theoremstyle{plain}

\newtheorem{theorem}[equation]{Theorem}

\newtheorem{lemma}[equation]{Lemma}

\theoremstyle{definition}

\theoremstyle{remark}

\newtheorem{remark}[equation]{Remark}

\title[Derived category of an Enriques surface in derived category of a Fano variety]%
{Embedding derived categories of Enriques surfaces\\[1ex]into derived categories of Fano varieties}
\author{Alexander Kuznetsov}
\address{{\sloppy
\parbox{0.9\textwidth}{
Algebraic Geometry Section, Steklov Mathematical Institute of Russian Academy of Sciences,\\
8 Gubkin str., Moscow 119991 Russia
\\[5pt]
The Poncelet Laboratory, Independent University of Moscow
\hfill\\[5pt]
Laboratory of Algebraic Geometry, National Research University Higher School of Economics, Russian Federation
}\bigskip}}
\email{akuznet@mi-ras.ru}
\date{}
\thanks{I was partially supported by the Russian Academic Excellence Project ``5-100''
and by the Program of the Presidium of the Russian Academy of Sciences~01 ``Fundamental Mathematics and
its Applications'' under grant PRAS-18-01.
I am grateful to A.~Fonarev, D.~Orlov and~C.~Shramov for inspiring discussions and the referee for useful comments.}

\begin{document}

\begin{abstract}
We show that the bounded derived category of coherent sheaves on a general Enriques surface can be realized
as a semiorthogonal component in the derived category of a smooth Fano variety
with diagonal Hodge diamond.
\end{abstract}

\maketitle

If a smooth projective variety $X$ over the field $\CC$ of complex numbers
has a full exceptional collection, then its Hodge diamond is \emph{diagonal}, i.e.,
\begin{equation*}
h^{p,q}(X) = 0
\qquad
\text{for $p \ne q$}.
\end{equation*}
It is natural to ask whether the converse is true.
A simple counterexample to this naive question is provided by an Enriques surface $S$ ---
its Hodge diamond looks like
\begin{equation*}
\begin{smallmatrix}
&& 1 \\
& 0 && 0 \\
0 && 10 && 0\\
& 0 && 0 \\
&& 1 
\end{smallmatrix},
\end{equation*}
so it is diagonal; on the other hand, its Grothendieck group $K_0(S)$ contains a 2-torsion class (see, for instance, \cite[Lemma~2.2]{GKMS}),
hence the derived category cannot be generated by a full exceptional collection by the next simple lemma.

\begin{lemma}[{cf.~\cite[\S3]{BP93}, \cite[Proposition~2.1(5)]{GKMS}}]
\label{lemma:no-ec}
Let $\cT$ be a triangulated category such that the Grothendieck group $K_0(\cT)$ contains a torsion class. 
Then $\cT$ does not admit a full exceptional collection.
\end{lemma}

\begin{proof}
Assume $\cT$ is generated by an exceptional collection of length~$n$.
Since the Grothendieck group is additive with respect to semiorthogonal decompositions, we have $K_0(\cT) \cong \ZZ^n$.
In particular, $K_0(\cT)$ is torsion free.
\end{proof}

The question that is a bit less naive --- whether a \emph{Fano variety} with diagonal Hodge diamond
necessarily has a full exceptional collection --- was asked by Alexey Bondal back in 1989.
This question was raised again in a recent paper~\cite{PS}.
The main goal of this note is to show that the answer is still negative, 
and again counterexamples can be constructed using Enriques surfaces.

To be more precise, we construct a smooth Fano variety $X$ such that its bounded derived category $\bD(X)$ of coherent sheaves
has a semiorthogonal decomposition whose components are several exceptional objects and $\bD(S)$, where $S$ is an Enriques surface. 
Thus, the Hodge diamond of $X$ is diagonal, but the Grothendieck group $K_0(X)$ contains a 2-torsion class (coming from $K_0(S)$),
hence~$\bD(X)$ does not have a full exceptional collection by Lemma~\ref{lemma:no-ec}.

In fact, we present two such constructions.

\bigskip

In the first construction, $S$ is a general Enriques surface from a certain divisorial
family in the moduli space of Enriques surfaces --- such~$S$ are called ``nodal Enriques surfaces'' or ``Reye congruences''.
By~\cite[Theorem~3.2.2]{C83} an Enriques surface $S$ of this type can be embedded into the Grassmannian~$\Gr(2,4)$,
and~\cite[Lemma~5.1]{IK} describes a resolution of its structure sheaf.

We consider the blowup
\begin{equation*}
M = \Bl_S(\Gr(2,4)).
\end{equation*}

\begin{theorem}
\label{theorem:reye}
The variety $M$ is a Fano $4$-fold with a semiorthogonal decomposition 
\begin{equation*}
\bD(M) = \langle \bD(S), E_1, \dots, E_6 \rangle,
\end{equation*}
where $E_1,\dots,E_6$ are exceptional bundles.
The Hodge diamond of $M$ is diagonal, but $K_0(M)$ contains a~$2$-torsion class;
in particular $\bD(M)$ does not have a full exceptional collection.
\end{theorem}
\begin{proof}
By~\cite[Lemma~5.2, 5.3]{IK} the variety $M$ can be embedded into the product $\Gr(2,4) \times \PP^3$
as the zero locus of a regular section of the rank-3 vector bundle $\Sym^2\cU^\vee \boxtimes \cO(1)$, where $\cU$
is the tautological vector bundle of the Grassmannian. 
The determinant of this vector bundle is isomorphic to $\cO(3) \boxtimes \cO(3)$, 
hence by adjunction formula $\omega_M^{-1} \cong (\cO(1) \boxtimes \cO(1))\vert_M$ is the restriction
of an ample line bundle, hence $M$ is a Fano 4-fold.

The semiorthogonal decomposition is given by the Orlov's blowup formula and the fact that $\bD(\Gr(2,4))$
is generated by an exceptional collection of length 6.
The Hodge diamond of $M$ looks like
\begin{equation*}
\begin{smallmatrix}
&&&& 1 \\
&&& 0 && 0 \\
&& 0 && 1 && 0 \\
& 0 && 0 && 0 && 0 \\
0 && 0 && 2 && 0 && 0\\
& 0 && 0 && 0 && 0 \\
&& 0 && 1 && 0 \\
&&& 0 && 0 \\
&&&& 1 
\end{smallmatrix}
\ +\ 
\begin{smallmatrix}
&& 1 \\
& 0 && 0 \\
0 && 10 && 0\\
& 0 && 0 \\
&& 1 
\end{smallmatrix}
\ =\ 
\begin{smallmatrix}
&&&& 1 \\
&&& 0 && 0 \\
&& 0 && 2 && 0 \\
& 0 && 0 && 0 && 0 \\
0 && 0 && 12 && 0 && 0\\
& 0 && 0 && 0 && 0 \\
&& 0 && 2 && 0 \\
&&& 0 && 0 \\
&&&& 1 
\end{smallmatrix},
\end{equation*}
a combination of the Hodge diamonds of $\Gr(2,4)$ and $S$, again thanks to the blowup representation.
The Grothendieck group is additive with respect to semiorthogonal decompositions, hence
\begin{equation*}
K_0(M) = K_0(S) \oplus \ZZ^6;
\end{equation*}
in particular the 2-torsion class in $S$ gives a 2-torsion class in $M$.
We conclude by Lemma~\ref{lemma:no-ec}.
\end{proof}

\bigskip

The second construction works for a general Enriques surface 
(i.e., corresponding to any point of an open subset in the moduli space of Enriques surfaces),
at the price that the corresponding Fano variety is 6-dimensional.

Let $V_1$ and $V_2$ be a pair of 3-dimensional vector spaces.
Consider the Veronese embeddings
\begin{equation*}
\PP(V_1) \hookrightarrow \PP(\Sym^2V_1) \hookrightarrow \PP(\Sym^2V_1 \oplus \Sym^2V_2),
\qquad
\PP(V_2) \hookrightarrow \PP(\Sym^2V_2) \hookrightarrow \PP(\Sym^2V_1 \oplus \Sym^2V_2),
\end{equation*}
and their join $\rJ(\PP(V_1),\PP(V_2)) \subset \PP(\Sym^2V_1 \oplus \Sym^2V_2)$.
This is a singular 5-dimensional variety, whose singularities are resolved by the projective bundle
\begin{equation*}
\bJ := \PP_{\PP(V_1) \times \PP(V_2)}(\cO(-2,0) \oplus \cO(0,-2)).
\end{equation*}
Indeed, denote by~$H_1$ and $H_2$ the pullbacks to $\bJ$ of the hyperplane classes of the two factors~$\PP(V_1)$ and~$\PP(V_2)$,
by~$H$ the Grothendieck relative class of the projectivization, and by $\pi \colon \bJ \to \PP(V_1) \times \PP(V_2)$ the projection.
Then the natural embedding
\begin{equation*}
\cO_\bJ(-H) \hookrightarrow \pi^*(\cO(-2,0) \oplus \cO(0,-2)) \hookrightarrow (\Sym^2V_1 \otimes \cO) \oplus (\Sym^2V_2 \otimes \cO)
\end{equation*}
defines a morphism $\bJ \to \PP(\Sym^2V_1 \oplus \Sym^2V_2)$ which contracts the divisors 
\begin{equation*}
\PP_{\PP(V_1) \times \PP(V_2)}(\cO(-2,0)) \subset \bJ,
\qquad
\PP_{\PP(V_1) \times \PP(V_2)}(\cO(0,-2)) \subset \bJ
\end{equation*}
onto the two Veronese surfaces 
$\PP(V_1) \hookrightarrow \PP(\Sym^2V_1 \oplus \Sym^2V_2)$ and $\PP(V_2) \hookrightarrow \PP(\Sym^2V_1 \oplus \Sym^2V_2)$,
and takes the fibers of $\pi$ to the lines joining the corresponding points of these.

Below we consider a global section of the vector bundle $\cO_\bJ(H)^{\oplus 3}$ on $\bJ$.
Note that 
\begin{equation*}
H^0(\bJ, \cO_\bJ(H)) \cong H^0(\PP(V_1) \times \PP(V_2), \cO(2,0) \oplus \cO(0,2)) \cong \Sym^2V_1^\vee \oplus \Sym^2V_2^\vee,
\end{equation*}
so such a section is given by a linear map 
\begin{equation*}
\phi \colon W \to \Sym^2V_1^\vee \oplus \Sym^2V_2^\vee
\end{equation*}
from a 3-dimensional vector space $W$.
We will denote the corresponding section also by $\phi$.

\begin{lemma}
The zero locus $S \subset \bJ$ of a general section $\phi$ of the vector bundle $\cO_\bJ(H)^{\oplus 3}$ on $\bJ$ is an Enriques surface.
A general Enriques surface can be obtained in this way.
\end{lemma}
\begin{proof}
Consider another projective bundle
\begin{equation*}
\widetilde\bJ \cong \PP_{\PP(V_1) \times \PP(V_2)}(\cO(-1,0) \oplus \cO(0,-1)).
\end{equation*}
It is isomorphic to the blowup of $\PP(V_1 \oplus V_2)$ along the union of two skew planes $\PP(V_1) \sqcup \PP(V_2) \subset \PP(V_1 \oplus V_2)$
with the exceptional divisors
\begin{equation*}
E_1 = \PP_{\PP(V_1) \times \PP(V_2)}(\cO(-1,0)) \subset \widetilde\bJ,
\qquad 
E_2 = \PP_{\PP(V_1) \times \PP(V_2)}(\cO(0,-1)) \subset \widetilde\bJ.
\end{equation*}
Denote by $\tH_1$ and $\tH_2$ the pullbacks to $\widetilde\bJ$ of the hyperplane classes of the two factors~$\PP(V_1)$ and~$\PP(V_2)$,
and by $\tH$ the Grothendieck relative class of the projectivization.
Then $E_1 \equiv \tH - \tH_2$ and $E_2 \equiv \tH - \tH_1$.

Consider the involution of the bundle $\cO(-1,0) \oplus \cO(0,-1)$ acting with weight $-1$ on the first summand and with weight $1$ on the second,
and the corresponding involution $\tau$ of $\widetilde\bJ$.
The fixed locus of $\tau$ is the union of the exceptional divisors $E_1 \sqcup E_2$, 
and the quotient $\widetilde\bJ/\tau$ is isomorphic to $\bJ$ with the quotient map~$f \colon \widetilde\bJ \to \bJ$ induced by the projection
\begin{equation*}
\Sym^2(\cO(-1,0) \oplus \cO(0,-1)) = 
\cO(-2,0) \oplus \cO(-1,-1) \oplus \cO(0,-2)
\twoheadrightarrow \cO(-2,0) \oplus \cO(0,-2).
\end{equation*}
Note also, that $\cO_{\widetilde\bJ}(2\tH) \cong f^*(\cO_\bJ(H))$,
and this induces an isomorphism 
\begin{equation*}
H^0(\widetilde\bJ, \cO_{\widetilde\bJ}(2\tH))^\tau \cong
H^0(\bJ, \cO_\bJ(H)) \cong \Sym^2V_1^\vee \oplus \Sym^2V_2^\vee
\end{equation*}
between the space of $\tau$-invariant global sections of $\cO_{\widetilde\bJ}(2\tH)$ 
and the space of global sections of $\cO_\bJ(H)$.
Therefore, the preimage 
\begin{equation*}
\tS := f^{-1}(S) \subset \widetilde\bJ
\end{equation*}
is the zero locus of a general $\tau$-invariant section of the vector bundle $\cO_\bJ(2\tH)^{\oplus 3}$.
We have
\begin{equation*}
K_\tS 
\equiv K_{\widetilde\bJ} + 6\tH
\equiv (- 3\tH_1 - 3\tH_2) + (\tH_1 + \tH_2 - 2\tH) + 6\tH 
\equiv 4\tH -2\tH_1 - 2\tH_2
\equiv 2E_1 + 2E_2.
\end{equation*}

Recall that $S$ is defined by a map $\phi \colon W \to \Sym^2V_1^\vee \oplus \Sym^2V_2^\vee$.
Clearly $\tS \cap E_i$ is equal to the intersection of three conics in $\PP(V_i)$ 
corresponding to the induced map $\phi_i \colon W \to \Sym^2V_i^\vee$,
hence is empty for a general choice of~$S$.
This shows that for a general $S$, the surface $\tS$ meets neither $E_1$ nor $E_2$, hence $K_\tS \equiv 0$.

Furthermore, it is easy to see that~$H^{1}(\tS,\cO_\tS) = 0$ (for instance, by using the Koszul resolution of $\cO_\tS$ on $\widetilde\bJ$), 
hence $\tS$ is a K3 surface.
As $\tS$ does not intersect the fixed locus $E_1 \sqcup E_2$ of $\tau$, the involution $\tau$ acts freely on $\tS$, hence
\begin{equation*}
S \cong \tS/\tau \subset \widetilde\bJ/\tau = \bJ
\end{equation*}
is an Enriques surface.

Finally, note that the surface $\tS$ defined above coincides with the surface $X$ in~\cite[Exercise~VIII.18]{B-book},
and the involution $\tau$ on $\tS$ coincides with the involution $\sigma$ in \emph{loc.\ cit.}
Therefore, the quotient $S = \tS/\tau$ is a general Enriques surface.
\end{proof}

Next we consider the product $\bJ \times \PP(W)$ 
that parametrizes the linear system of sections of $\cO_\bJ(H)$, cutting out $S$ in $\bJ$.
Denote by $H'$ the hyperplane class of this $\PP(W)$ and let 
\begin{equation*}
X \subset \bJ \times \PP(W)
\end{equation*}
be the universal divisor from the linear system, i.e., the zero locus on $\bJ \times \PP(W)$
of the global section of the line bundle $\cO_\bJ(H) \boxtimes \cO(H')$ corresponding to the map~$\phi$.

\begin{theorem}
\label{theorem:general}
The variety $X$ is a Fano $6$-fold with a semiorthogonal decomposition 
\begin{equation*}
\bD(X) = \langle \bD(S), F_1, \dots, F_{36} \rangle,
\end{equation*}
where $F_1,\dots,F_{36}$ are exceptional bundles.
The Hodge diamond of $X$ is diagonal, but $K_0(X)$ contains a~$2$-torsion class;
in particular $\bD(X)$ does not have a full exceptional collection.
\end{theorem}
\begin{proof}
The canonical class of $X$ is equal to
\begin{equation*}
K_X 
= K_\bJ + K_{\PP(W)} + (H + H')
= (- 3H_1 - 3H_2) + (2H_1 + 2H_2 - 2H) - 3H' + (H + H') 
= - H_1 - H_2 - H - 2H'.
\end{equation*}
Let us show that $-K_X$ is ample.
Clearly, for this it is enough to check that $H + H_1 + H_2$ is ample on $\bJ$.
By~\cite[Proposition~3.2]{H66} this is equivalent to ampleness 
of its pushforward $\cO(3H_1 + H_2) \oplus \cO(H_1 + 3H_2)$ on $\PP(V_1) \times \PP(V_2)$, 
which follows from~\cite[Proposition~2.2]{H66} and from ampleness of the summands.
We conclude that~$X$ is a Fano 6-fold.

The map $X \to \bJ$ has general fiber $\PP^1$, and over the surface $S \subset \bJ$ the fibers jump to $\PP^2$.
Therefore,
\begin{equation*}
\bD(X) = \langle \bD(S), \bD(\bJ), \bD(\bJ) \rangle,
\end{equation*}
either by~\cite[Theorem~8.8]{K-hpd}, or by~\cite[Proposition~2.10]{O}.
Since $\bJ$ is a $\PP^1$-bundle over $\PP^2\times \PP^2$, its derived category is generated by $3\cdot 3\cdot 2 = 18$ exceptional bundles,
hence we obtain the required semiorthogonal decomposition for $\bD(X)$.
Finally, the Hodge diamond of $X$ looks like
\begin{equation*}
\begin{smallmatrix}
&&&&&& 1 \\
&&&&& 0 && 0 \\
&&&& 0 && 4 && 0 \\
&&& 0 && 0 && 0 && 0 \\
&& 0 && 0 && 8 && 0 && 0\\
& 0 && 0 && 0 && 0 && 0 && 0 \\
0 && 0 && 0 && 10 && 0 && 0 && 0\\
& 0 && 0 && 0 && 0 && 0 && 0 \\
&& 0 && 0 && 8 && 0 && 0\\
&&& 0 && 0 && 0 && 0 \\
&&&& 0 && 4 && 0 \\
&&&&& 0 && 0 \\
&&&&&& 1 
\end{smallmatrix}
\ +\ 
\begin{smallmatrix}
&& 1 \\
& 0 && 0 \\
0 && 10 && 0\\
& 0 && 0 \\
&& 1 
\end{smallmatrix}
\ =\ 
\begin{smallmatrix}
&&&&&& 1 \\
&&&&& 0 && 0 \\
&&&& 0 && 4 && 0 \\
&&& 0 && 0 && 0 && 0 \\
&& 0 && 0 && 9 && 0 && 0\\
& 0 && 0 && 0 && 0 && 0 && 0 \\
0 && 0 && 0 && 20 && 0 && 0 && 0\\
& 0 && 0 && 0 && 0 && 0 && 0 \\
&& 0 && 0 && 9 && 0 && 0\\
&&& 0 && 0 && 0 && 0 \\
&&&& 0 && 4 && 0 \\
&&&&& 0 && 0 \\
&&&&&& 1 
\end{smallmatrix},
\end{equation*}
a combination of the Hodge diamonds of $\bJ\times\PP^1$ and $S$.
The Grothendieck group is additive with respect to semiorthogonal decompositions, hence
\begin{equation*}
K_0(X) = K_0(S) \oplus \ZZ^{36};
\end{equation*}
in particular the 2-torsion class in $S$ gives a 2-torsion class in $X$.
We conclude by Lemma~\ref{lemma:no-ec}.
\end{proof}

\begin{remark}
The embedding of the derived category of a general Enriques surface into a Fano variety, constructed in Theorem~\ref{theorem:general},
solves for them so-called ``Fano-visitor problem'' suggested by Alexey Bondal in 2011, see~\cite{BBF,KKLL,KL15,N,KF}. 
Note that a similar embedding of $\bD(S)$ into the derived category of a Fano \emph{orbifold} was constructed in~\cite[6.2.3]{KL15}.
\end{remark}


\end{document}